\documentclass[12pt]{article}
\input epsf.tex

\usepackage{amsmath}
\usepackage{amsthm}
\usepackage{amsfonts}
\usepackage{amssymb}
\usepackage{graphicx}
\usepackage{latexsym}

\usepackage{amsmath,amsthm,amsfonts,amssymb}

\usepackage{epsfig}

\usepackage{amssymb}
\usepackage[all]{xy}
\xyoption{poly}
\usepackage{fancyhdr}
\usepackage{wrapfig}

\theoremstyle{plain}
\newtheorem{thm}{Theorem}[section]

\newtheorem{lem}[thm]{Lemma}
\newtheorem{cor}[thm]{Corollary}

\theoremstyle{definition}
\newtheorem{defn}{Definition}
\theoremstyle{remark}
\newtheorem{remark}{Remark}

\advance \topmargin by -\headheight
\advance \topmargin by -\headsep
\evensidemargin \oddsidemargin
\def\cB{{\mathcal B}}

\def\cc{{\curvearrowright}}

\def\F{{\mathbb F}}
\def\cF{{\mathcal F}}

\def\Fix{\textrm{Fix}}

\def\N{{\mathbb N}}

\def\R{{\mathbb R}}

\def\cR{\mathcal R}

\def\chix{{\raise.5ex\hbox{$\chi$}}}

\def\Z{{\mathbb Z}}

\newcommand\Aut{\operatorname{Aut}}

\newcommand\Hom{\operatorname{Hom}}

\newcommand\IRS{\operatorname{IRS}}

\newcommand\res{\upharpoonright}
\newcommand\RG{{\operatorname{RG}}}

\newcommand\Stab{\operatorname{Stab}}
\newcommand\Sub{\operatorname{Sub}}

\newcommand\SC{\operatorname{SC}}
\newcommand\Sch{\operatorname{Sch}}

\begin{document}
\title{Invariant random subgroups of the free group}
\author{Lewis Bowen\footnote{supported in part by NSF grant DMS-0968762, NSF CAREER Award DMS-0954606 and BSF grant 2008274} \\ University of Texas at Austin}
\maketitle
\begin{abstract}
Let $G$ be a locally compact group. A random closed subgroup with conjugation-invariant law is called an {\em invariant random subgroup} or IRS for short. We show that each nonabelian free group has a large ``zoo'' of IRS's. This contrasts with results of Stuck-Zimmer which show that there are no non-obvious IRS's of higher rank semisimple Lie groups with property (T).
\end{abstract}

\noindent
{\bf Keywords}: Invariant random subgroups, ergodic equivalence relations\\
{\bf MSC}:37A15, 37A20 \\

\tableofcontents

\section{Introduction}

Let $G$ be a locally compact group. A random closed subgroup with conjugation-invariant law is called an {\em invariant random subgroup} or IRS for short. These objects arise naturally from the study of group actions. Indeed, suppose $G$ acts on a standard Borel space $X$. For $x\in X$, let $\Stab(x)=\{g \in G:~ gx=x\}$. If $\mu$ is a $G$-invariant probability measure on $X$ and $x\in X$ is chosen at random with law $\mu$ then $\Stab(x)$ is an IRS. By  \cite[Proposition 13]{AGV12} every IRS occurs from this type of construction.

There has been a recent increase in studies of the action of $G$ on its space of subgroups \cite{AGV12, Vo12, AB+11, Ve11, Sa11, Gr11, Ve10, Bo10, BS06, DS02, GS99, SZ94}. Perhaps the deepest result in the subject is the Stuck-Zimmer Theorem \cite{SZ94}: there are no non-obvious IRS's of higher rank simple Lie groups. More precisely, every ergodic IRS is induced from a lattice subgroup. Another nice result in this area is a complete classification IRS's of the infinite symmetric group due to A. Vershik \cite{Ve11}. Our main goal is to show that, by contrast, there is a large zoo of IRS's of any finitely generated nonabelian free group. 

There are four main results of this paper: (1) the space of laws of totally non-free IRS's is residual in the space of all laws of IRS's,  (2)  the space of laws of ergodic infinite index IRS's of the free group is homeomorphic to Hilbert space, (3) every ergodic aperiodic pmp equivalence relation of cost $<r$ is represented by the action of $\F_r$ on its space of subgroups and (4) that every symbolic action of $\F_r$ can be encoded as a sub-action of $\F_r$ on its space of subgroups.

To state these results precisely, we need to introduce some notation. Let $\Sub(G)$ denote the space of all closed subgroups of a locally compact group $G$ with the topology of uniform convergence on compact sets. $G$ acts on $\Sub(G)$ by $g \cdot H := gHg^{-1}$ and $\IRS(G)$ denotes the space of all conjugation-invariant Borel probability measures on $\Sub(G)$. We always consider $\IRS(G)$ with the weak* topology. We let $\IRS_i(G)$ denote those measures $\nu \in \IRS(G)$ such that $\nu$-a.e. $K \in \Sub(G)$ has infinite index.

\subsection{Totally non-free actions}
Let $G \cc (X,\cB,\mu)$ be a probability-measure-preserving (pmp) action. For any Borel set $H \subset G$, let $\Fix(H)=\{x \in X:~ hx =x ~\forall h \in H\}$. A probability measure preserving action $G \cc (X,\cB,\mu)$ is {\em totally non-free} if the collection $\{\Fix(H):~ H $ a Borel subset of $G\} \subset \cB$ generates the Borel sigma-algebra $\cB$ (up to sets of measure zero). The study of such actions was initiated in \cite{Ve11} where a classification of such actions is called for. Our first result:\\

\noindent {\bf Corollary \ref{cor:tnf}}
{\em The set of measures $\mu \in \IRS_{i}(\F_r)$ such that $\F_r \cc (\Sub(\F_r),\mu)$ is totally non-free is residual in $\IRS_{i}(\F_r)$.}

\subsection{The simplex $\IRS_i(G)$}

Recall  that a convex closed metrizable subset $K$ of a locally convex linear
space is a {\em simplex} if each point in $K$ is the barycenter of a unique probability
measure supported on the subset $\partial_e K$ of extreme points of  $K$. In this case, $K$ is called a {\em Poulsen simplex} if $\partial_e K$ is dense in $K$. It is called a {\em Bauer simplex} if $\partial_e K$ is closed.

For example, an old result states that the space of all shift-invariant Borel probability measures on $\{0,1\}^\Z$ is a Poulsen simplex \cite[page 45]{Ol80}.  By \cite{LOS78} there is a unique Poulsen simplex up to affine homeomorphism and its set of extreme points is homeomorphic to the Hilbert space $l^2$. On the other hand, there are uncountably many non-isomorphic Bauer simplices. For example, let $X$ be any compact metrizable space. Then the space $P(X)$ of all Borel probability measures on $X$ is a Bauer simplex and $\partial_e P(X)$ is homeomorphic to $X$.

Let $\IRS_e(G) \subset \IRS(G)$ denote the ergodic measures and $\IRS_{fi}(G) \subset \IRS_e(G)$ those measures which are supported on the conjugacy class of a single finite-index subgroup. If $G$ is finitely generated, then each of its finite-index subgroups are finitely generated and because of this, each finite-index subgroup is isolated in the space of subgroups. This implies each element of $\IRS_{fi}(G)$ is isolated in $\IRS_e(G)$. In particular, $\IRS_e(G)$ cannot be connected if $G$ has proper finite-index subgroups.

Let $\IRS_{ie}(G) :=\IRS_e(G) \setminus \IRS_{fi}(G)$ and $\IRS_{i}(G)$ denote the closed convex hull of $\IRS_{ie}(G)$ (this agrees with our previous definition). Our second result:\\ 

\noindent {\bf Theorem \ref{thm:topology}}.
{\em If $G$ is a nonabelian free group, then $\IRS_{i}(G)$ is a Poulsen simplex. In particular, $\IRS_{ie}(G)$ is a dense $G_\delta$ subset of $\IRS_i(G)$ and $\IRS_{ie}(G)$ is homeomorphic to the Hilbert space $l^2$. }


\begin{remark}
In order to discuss the case of locally compact groups, let $\IRS_c(G) \subset \IRS_e(G)$ be the set of measures which are supported on the conjugacy class of a single cocompact subgroup. Let $\IRS_{nce} = \IRS_e(G) \setminus \IRS_c(G)$ and $\IRS_{nc}(G)$ be the closed convex hull of $\IRS_{nce}(G)$. Using the ideas of the proof of Theorem \ref{thm:topology}, it can be shown that $\IRS_{nc}(PSL_2(\R))$ is also a Poulsen simplex. By constract, the main results of \cite{GW97} imply that if $G$ has property (T) then $\IRS_{nc}(G)$ and $\IRS(G)$ are Bauer simplices.
\end{remark}

\begin{remark}
Grigorchuk's {\em space of $r$-generated marked groups} \cite{Gr84} is (canonically isomorphic with) the space of normal subgroups of the free group of rank $r$ with the topology of uniform convergence on compact subsets (also known as the Chabauty topology). It can be viewed as a subspace of $\IRS_e(\F_r)$ (where $\F_r$ denotes the free group of rank $r$). Namely, for each normal subgroup $N$, let $\delta_N$ be the Dirac probability measure supported on $\{N\}$. Then $\{\delta_N:~N\vartriangleleft \F_r\} \subset \IRS_e(\F_r)$ is a copy of the space of $r$-generated marked groups. It is known that the space of marked groups contains interesting isolated points \cite{CGP07}. In particular, there are infinite index normal subgroups which are isolated in the space of all normal subgroups. However, these points are {\em not } isolated in $\IRS_e(\F_r)$ because $\IRS_{ie}(\F_r)$ is pathwise connected by Theorem \ref{thm:topology}.

\end{remark}

\subsection{Measured equivalence relations}
To explain the next result, we need to recall some notions from the theory of measured equivalence relations. So let $(X,\mu)$ be a standard Borel probability space and $E \subset X \times X$ be a discrete Borel equivalence relation such that $\mu$ is $E$-invariant (i.e., if $\phi:X \to X$ is a Borel automorphism whose graph is contained in $E$ then $\phi_*\mu=\mu$). We refer to the triple $(X,\mu,E)$ as a {\em discrete probability-measure-preserving (pmp) equivalence relation}. Two discrete pmp equivalence relations $(X_i,\mu_i,E_i)$ (for $i=1,2$) are {\em isomorphic} if there exist conull sets $X'_i \subset X_i$ and a measure-space isomorphism $\phi:(X'_1,\mu_1)\to (X'_2,\mu_2)$ such that $(x,y) \in E_1 \Leftrightarrow (\phi(x),\phi(y)) \in E_2$. In particular, we only require that $\phi$ is defined on a set of full measure. 
 
Let $E_G$ denote the orbit equivalence relation on $\Sub(G)$: $E_G:=\{ (K, gKg^{-1}):~ K \in \Sub(G), g\in G\}$. Our third main result:\\

\noindent {\bf Corollary \ref{cor:universal}}.
{\em If $G$ is a free group of rank $r$ and $(X,\mu,E)$ is an ergodic aperiodic discrete pmp equivalence relation with cost(E) $<r$ then there exists an invariant measure $\lambda \in \IRS(G)$ such that $(\Sub(G), \lambda, E_G)$ is isomorphic to $(X,\mu,E)$. Moreover, there is an action $G \cc X$ such that $E=\{(x,gx):~x\in X, g\in G\}$ is the orbit-equivalence relation and if $\Stab:X \to \Sub(G)$ is the stabilizer map $\Stab(x)=\{g\in G:~gx=x\}$ then $\Stab$ is a measure-conjugacy from $G\cc (X,\mu)$ to $G\cc (\Sub(G),\lambda)$.}

\begin{remark}
The authors of \cite{IKT09}  and I. Epstein proved that if $G$ is any nonamenable group then it is impossible to classify up to countable structures the free mixing probability measure-preserving actions of $G$ up to orbit-equivalence. Note that if $G$ has a generating set with less than $r$ generators then all of its orbit equivalence relations have cost $<r$.

\end{remark}

\begin{remark}
A well-known result of \cite{FM77} states that any discrete pmp equivalence relation $(X,\mu,E)$ can be realized as the orbit-equivalence relation for the action of a countable group $G$. In other words, given $(X,\mu,E)$ there is a countable group $G$ and a measure-preserving action $G \cc (X,\mu)$ such that $(x,y) \in E \Leftrightarrow \exists g\in G$ such that $gx=y$. The authors then asked whether there exists a group $G$ and an {\em essentially free} action $G \cc (X,\mu)$ generating $E$. The negative answer to this question was obtained by Furman \cite{Fu99}. By contrast, Corollary \ref{cor:universal} implies that if cost$(E)<r$ then there is a {\em totally non-free} action of $\F_r$, the free group of rank $r$, generating the equivalence relation. 
\end{remark}

\subsection{Symbolic dynamics}
Finally, we provide a general method for constructing invariant subspaces and measures on $\Sub(\F_r)$ from symbolic actions of $\F_r$ (where $\F_r$ is the rank $r$ free group). Briefly the result states the following. Let $\F'_r$ be the subgroup of $\F_r$ generated by $\{s_1^2,s_2,\ldots, s_r\}$. If $K$ is any finite set and $X \subset K^{\F_r}$ any shift-invariant closed subset then there exists an $\F_r$-invariant closed subset $Y \subset \Sub(\F_r)$ and an $\F'_r$-invariant closed subset $Z \subset Y$ such that $\F_r \cc X$ is topologically conjugate to $\F'_r \cc Z$ via the obvious isomorphism from $\F_r$ to $\F'_r$. Moreover, $Y$ is covered by finitely many translates of $Z$. In this sense, we can embed the dynamics of any symbolic action into the dynamics of the conjugation-action on a space of subgroups. See Theorem \ref{thm:encoding} for complete details.

{\bf Acknowledgements}. Thanks to Miklos Abert, Yair Glasner, Eli Glasner and Anatoly Vershik for stimulating questions on this topic and to Rostislav Grigorchuk for helpful comments. Thanks to Benjy Weiss for pointing out errors in a previous version. I'm grateful to the anonymous reviewers for catching several errors. Also thanks to Alessandro Carderi for catching a mistake in a previous version.


\section{Preliminaries}

\subsection{Schreier coset graphs}\label{sec:Schreier}
Let $L$ be a finite set. A {\em rooted $L$-labeled digraph} consists of a graph whose edges are labeled with labels in $L$ and a distinguished vertex called the {\em root}. Some of the edges may be directed while others may be undirected. We allow loops and multiple edges. Two rooted $L$-labeled digraphs $\Gamma,\Gamma'$ are {\em isomorphic} if there exists a graph isomorphism between them which preserves directions and labels. If such an isomorphism also preserves roots then we say $\Gamma$ and $\Gamma'$ are {\em root-isomorphic}.

Let $\RG(L,d)$ denote the set of all root-isomorphism classes of connected rooted $L$-labeled digraphs such that every vertex has degree at most $d$. This is a compact metrizable space: indeed, given $\Gamma,\Gamma' \in \RG(L,d)$ we let $d(\Gamma,\Gamma') = \frac{1}{n+1}$ where $n$ is the smallest nonnegative integer such that the ball of radius $n$ centered at the root in $\Gamma$ is root-isomorphic to the ball of radius $n$ centered at the root in $\Gamma'$ (as rooted labeled digraphs). This distance function makes $\RG(L,d)$ into a compact metric space.

Our main application of the construction above is to Schreier coset graphs of $\F_r$, the free group of rank $r$. So let $S=\{s_1,\ldots, s_r\}$ be a free generating set for $\F_r$. Given a subgroup $K<\F_r$, let $\Sch(K,S) \in \RG(S,2r)$ be the rooted labeled digraph with vertex set $K\backslash \F_r$, directed $s$-labeled edges $(Kg,Kgs)$ (for every $Kg \in K\backslash \F_r$ and $s\in S$) and root vertex $K$. By abuse of language, we will not distinguish between $\Sch(K,S)$ and its root-isomorphism class. Let $\SC(\F_r, S)\subset \RG(S,2r)$ be the set of all Schreier coset graphs of subgroups of $\F_r$. The map $K \in \Sub(\F_r) \mapsto \Sch(K,S) \in \SC(\F_r,S)$ is a homeomorphism. Also $\F_r$ acts on $\SC(\F_r,S)$ by $g\Sch(K,S):=  \Sch(gKg^{-1},S)$ and the homeomorphism between $\Sub(\F_r)$ and $\SC(\F_r,S)$ is equivariant with respect to this action. Moreover, $\Sch(gKg^{-1},S)$ is isomorphic to $\Sch(K,S)$ even though they are typically not root-isomorphic.  Precisely, $\Sch(gKg^{-1},S)$ is root-isomorphic to $\Sch(K,S)$ after moving the root in $\Sch(K,S)$ from $K$ to $Kg^{-1}$. The paper \cite{AGV12} further develops this correspondence between $\Sub(\F_r)$ and $\SC(\F_r,S)$.

\subsection{Measured equivalence relations}

Let $(X,\mu)$ be a standard Borel probability space and $E \subset X \times X$ be a Borel equivalence relation. Then $E$ is {\em discrete} if every $E$-class is countable. We use \cite{K} as a general reference.

Define measures $\mu_L, \mu_R$ on $E$ by
$$\mu_L(F) = \int | F \cap \pi_L^{-1}(x)| ~d\mu(x), \quad \mu_R(F) = \int | F \cap \pi_R^{-1}(x)| ~d\mu(x)$$
where $\pi_L:E \to X, \pi_R: E \to X$ are the left and right projection maps. The measure $\mu$ is {\em $E$-invariant} if $\mu_L=\mu_R$.

The full group of $E$, denoted $[E]$, is the group of all (equivalence classes of) invertible Borel transformations $f$ such that $\textrm{graph}(f)=\{(x,fx):~x\in X\}\subset E$. Two transformations are equivalent if they agree on a conull subset. By \cite[Proposition 3.2]{K}, $[E]$ with the uniform metric, defined by 
$$d_u(\phi,\psi) = \mu(\{x\in X:~\phi(x)\ne \psi(x)\}),$$
is a Polish group. It is a standard exercise to show that $\mu$ is $E$-invariant if and only if $\phi_*\mu=\mu$ for every $\phi \in [E]$. 

A subset $Y \subset X$  is {\em $E$-saturated} if $Y$ is a union of $E$-classes. We say $\mu$ is {\em $E$-ergodic} if $\mu(Y)\in \{0,1\}$ for every measurable $E$-saturated set $Y$.

Measured equivalence relations arise from actions of groups: if $G$ is a countable group and $G \cc (X,\mu)$ a measure-class-preserving action then $E=\{(x,gx):~x\in X, g\in G\}$ is a discrete Borel equivalence relation. Moreover, the action of $G$ is measure-preserving if and only if $\mu$ is $E$-invariant and the action of $G$ is ergodic if and only if $\mu$ is $E$-ergodic.

\section{Totally Non-Free Actions}\label{sec:tnf}


For $K < G$ let $N_G(K):=\{g \in G:~gKg^{-1}=K\}$ denote the normalizer of $K$ and $\IRS_n(G)$ be the set of all measures $\eta \in \IRS(G)$ such that $\eta(\{K \in \Sub(G):~ N_G(K)=K\})=1$.

\begin{lem}
For any $\mu \in \IRS_n(G)$, the action $G \cc (\Sub(G), \mu)$ is totally non-free.
\end{lem}

\begin{proof}
Assuming $\mu \in \IRS_n(G)$, for any Borel set $H \subset G$, $\mu(\Fix(H) \vartriangle \{ K \in \Sub(G):~ H \subset K\}) = 0$ where $\vartriangle$ denotes symmetric difference. Therefore, for any Borel set $H \subset G$, $\{K \in \Sub(G):~ H \subset K\}$ is contained in the sigma-algebra generated by the sets $\{\Fix(H):~H $ a Borel subset of $G\}$ (up to measure zero). This clearly generates the Borel sigma-algebra.


\end{proof}
Recall that $\F_r$ denotes the free group of rank $r$.
\begin{lem}\label{lem:normalizer}
For any $r \ge 2$, $\IRS_n(\F_r) \cap \IRS_{i}(\F_r)$ is dense in $\IRS_{i}(\F_r)$. 
\end{lem}

\begin{proof}


Let $\eta \in \IRS_i(\F_r)$. We will construct measures $\eta_p \in \IRS_i(\F_r) \cap \IRS_n(\F_r)$ for $p \in (0,1)$ such that $\eta_p \to \eta$ as $p\searrow 0$. 

Let $K \in \Sub(\F_r)$ be random with law $\eta$.  Let $u_p$ be the probability measure on $\{0,1,\ldots, r\}$ given by
\begin{displaymath}
u_p(\{i\}) = \left\{\begin{array}{cc}
1-p & i=0 \\
p/r & i >0
\end{array}\right.\end{displaymath}
Choose a function $x:K\backslash \F_r \to \{0,1\ldots, r\}$ by: $x(K)$ is random with law $u_q$ where $q = \frac{3p}{1+2p}$ and $x(Kg)$ is random with law $u_p$ if $Kg \ne K$. We require that the random variables $\{x(Kg):~Kg \in K\backslash \F_r\}$ are independent.


We construct a random rooted labeled digraph $\Gamma_p$ from the Schreier coset graph $\Sch(K,S)$ as follows. The vertex set of $\Gamma_p$ is equal to the disjoint union of $x^{-1}(0)$ with $x^{-1}(i) \times \{0,1,2\}$ for $i=1,\ldots, r$. Let $i \in \{1,\ldots, r\}$. Then there is an $s_i$-labeled edge in $\Gamma_p$ from
\begin{itemize}
\item $Kg$ to $Kgs_i$ whenever $x(Kg)=x(Kgs_i)=0$,
\item $Kg$ to $(Kgs_i,0)$ whenever $x(Kg)=0$ and $x(Kgs_i) >0$,
\item $(Kg,2)$ to $Kgs_i$ whenever $x(Kg)>0$ and $x(Kgs_i)=0$,
\item $(Kg, 2)$ to $(Kgs_i,0)$ whenever $x(Kg)>0$ and $x(Kgs_i)>0$,
\item $(Kg,0)$ to $(Kg,2)$ whenever $x(Kg)=i$,
\item $(Kg,1)$ to $(Kg,1)$ whenever $x(Kg)=i$,
\item $(Kg,0)$ to $(Kg,1)$ whenever $x(Kg)\ne i$ and $x(Kg)>0$,
\item $(Kg,1)$ to $(Kg,2)$ whenever $x(Kg)\ne i$ and $x(Kg)>0$.
\end{itemize}
See Figure \ref{fig:normalizer}. We determine the root of $\Gamma_p$ as follows. If $x(K)=0$ then let $K$ be the root of $\Gamma_p$. If $x(K)>0$ then let $(K,i)$ be the root of $\Gamma$ where $i\in \{0,1,2\}$ is random with the uniform distribution. 

 \begin{figure}[htb]
\begin{center}
\ \psfig{file=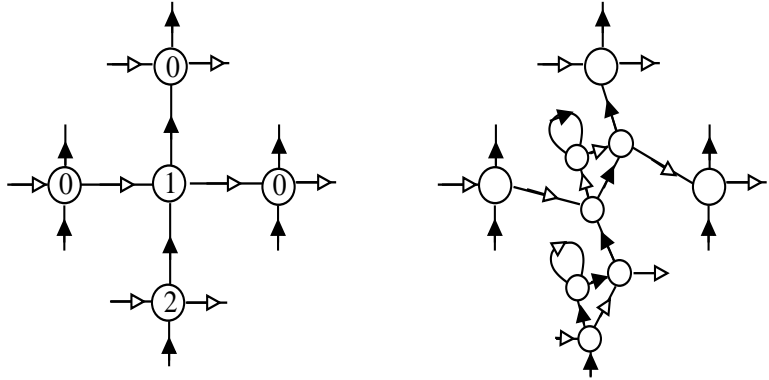,height=2in,width=4in}
\caption{The element $x$ is depicted on the left as a function on the vertices of the Schreier coset graph $K \backslash \F_2$. The Schreier coset graph $K_p \backslash \F_r$ is depicted on the right. Black arrows represent $s_1$ and white arrows represent $s_2$. }
\label{fig:normalizer}
\end{center}
\end{figure}

We claim that $\Gamma_p$ is the Schreier coset graph of a subgroup of $\F_r$. It suffices to show that for every vertex $v$ of $\Gamma_p$ and every $s\in S$ there is exactly one $s$-labeled edge directed into $v$ and one directed out of $v$. We leave this easy verification to the reader.

Let $K_p$ be the subgroup of $\F_r$ with $\Sch(K_p,S) \simeq \Gamma_p$. Let $\eta_p$ be the law of $K_p$. Because $\Gamma_p$ limits on $\Sch(K,S)$ in law as $p\searrow 0$, it follows that $\lim_{p \to 0+}\eta_p=\eta$. Because $\Sch(K,S)$ is infinite a.s., $\Gamma_p$ is also infinite a.s. So it suffices to show that $\eta_p$ is invariant and $N_{\F_r}(K_p) = K_p$ a.s.

\noindent {\bf Self-normalizing}.

Here we show that $N_{\F_r}(K_p) = K_p$ a.s. In general, if $T$ is a labeled digraph let $\Aut(T)$ denote the group of all automorphisms of $T$. If $T$ is also rooted, then we ignore the root (so an element of $\Aut(T)$ is not required to preserve the root). Observe that for any subgroup $K<\F_r$, $K \backslash N_{F_r}(K)$ is isomorphic to $\Aut(\Sch(K,S))$ by the map $Kn \mapsto (Kg \mapsto nKg = Kng)$. So it suffices to show that $\Aut(\Gamma_p)$ is trivial a.s.

Let $\Aut_*(\Gamma_p)$ be the set of all graph-automorphisms of $\Gamma_p$ that are required to preserve all edge-labels and directions {\em except} for the edge-labels on edges of the form $( (Kg,i), (Kg,j))$ for some $i,j\in \{0,1,2\}$. We consider two such automorphisms to be equivalent if they induce the same map on the vertex set. There are only countably many equivalence classes of automorphisms in $\Aut_*(\Gamma_p)$ because each one is completely determined by where it sends the root vertex.  For example, if such an automorphism $\phi$ maps $(Kg,0)$ to $(Kf,0)$ then it must map $(Kg,j)$ to $(Kf,j)$ for all $j \in \{0,1,2\}$.

So it suffices to show that any nontrivial $\phi \in \Aut_*(\Gamma_p)$ is not in $\Aut(\Gamma_p)$ a.s. To see this, let $U$ be the set of all vertices of $\Gamma_p$ of the form $(Kg,1)$ and let $U'$ be the set of all vertices of $\Gamma_p$ that have a single loop. Then $U \subset U'$ and $\phi$ preserves $U'$. In particular, if $u \in U$ then the unique loop based at $\phi(u)$ must have the same label as the loop at $u$. But this occurs with probability at most $\max(p/r, \frac{3p}{r(1+2p)})<1$. Because $p>0$, $U$ is infinite. So the event that the loops at $\phi(u)$ has the same label as $u$ for all $u\in U$ occurs with probability zero. So $\Aut(\Gamma_p)$ is trivial a.s. which implies $N_{\F_r}(K_p) = K_p$ a.s. 

\noindent {\bf Invariance}.

We claim that $\eta_p$ is invariant. To see this, let $\SC'$ be the set of all pairs $(\Gamma',x')$ where $\Gamma' \in \SC(\F_r,S)$ is a Schreier coset graph and $x'$ is a map from the vertices of $\Gamma'$ to $\{0,1,\ldots, r\}$. Let $E'$ be the equivalence relation on $\SC'$ determined by: $(\Gamma',x')E'(\Gamma'',x'')$ if there is a graph isomorphism $\phi$ from $\Gamma'$ to $\Gamma''$ preserving all directions, labels on the edges and taking $x'$ to $x''$ (so $x''(\phi(v))=x'(v)$). In other words, $(\Gamma',x')$ and $(\Gamma'',x'')$ are isomorphic after forgetting about their roots. This is a discrete Borel equivalence relation and if $\mu_p$ is the law of $(\Sch(K,S),x)$ (with $K,x$ as above) then $\mu_p$ is $E$-invariant. 

Now let $Y$ be the set of all $(\Gamma',x') \in \SC'$ such that $x'(v)>0$ where $v$ is the root of $\Gamma'$. Let $Z$ be the disjoint union of $\SC'\setminus Y$ with $Y \times \{0,1,2\}$. Let $\pi: Z \to \SC'$ be the projection map:
\begin{displaymath}
\pi(z) = \left\{\begin{array}{ll}
z & \textrm{if } z \in \SC'\setminus Y \\
y & \textrm{if } z = (y, i) \in Y \times \{0,1,2\}
\end{array}\right.
\end{displaymath}
Let $E''$ be the equivalence relation on $Z$ given by: $z_1 E'' z_2 \Leftrightarrow \pi(z_1)E' \pi(z_2)$. Let $\mu'_p$ be the measure on $Z$ given by $\mu'_p \upharpoonright \SC'\setminus Y = \mu_p \upharpoonright \SC'\setminus Y$ and $\mu'_p \upharpoonright Y \times \{0,1,2\} = (\mu_p \upharpoonright Y) \times \tau$ where $\tau$ denotes the uniform probability measure on $\{0,1,2\}$. Then $\frac{1}{1+2p}\mu'_p$ is an $E''$-invariant probability measure. 

Let $\phi:Z \to \Sub(\F_r)$ be the map constructed above. More precisely, an element $z \in Z$ corresponds to a point $(\Gamma',x') \in \SC'$ and possibly an index $i\in \{0,1,2\}$. Then $\phi(z)$ is the subgroup whose Schreier coset graph is obtained from $(\Gamma',x')$ in the manner depicted by Figure \ref{fig:normalizer} (so $\phi_*\mu'_p = \eta_p$). The index $i$ (if it is present) corresponds to a choice of one of the three ``new'' vertices. 

Observe that $\phi$ is {\em class-bijective}: for every $E''$-equivalence class $[z]$, the restriction of $\phi$ to $[z]$ is bijective onto the conjugacy class of $\phi$ (for a.e. $z$). This uses the fact that $N_{\F_r}(K_p) = K_p$ a.s. Because $\mu'_p$ is $E''$-invariant, it now follows that $\phi_*\mu'_p = \eta_p$ is $\F_r$-invariant.

\end{proof}

\begin{lem}
The subset $\IRS_n(\F_r)$ is a $G_\delta$-subset of $\IRS(\F_r)$.
\end{lem}

\begin{proof}

Given a finite set $F \subset \F_r$ and an element $g \in \F_r$, let $X(F,g)$ be the set of $K \in \Sub(\F_r)$ such that $g(F \cap K)g^{-1} \subset K$.  Let $Y(g)=\{K \in \Sub(\F_r):~ g \notin K\}$. Both $X(F,g)$ and $Y(g)$ are clopen sets. Let $Z(g) = Y(g) \cap \bigcap_F X(F,g)$. Then $Z(g)$ is a closed set. Note that $K \in Z(g)$ if and only if $g \in N_{\F_r}(K) \setminus K$. 

Let $O(g,\epsilon)$ be the set of all measures $\mu$ in $\IRS(\F_r)$ such that $\mu(Z(g))<\epsilon$.  Then $O(g,\epsilon)$ is open and
$$\IRS_n(\F_r) =\cap_{g\in \F_r}  \cap_{m=1}^\infty  O(g,1/m).$$



\end{proof}

\begin{cor}\label{cor:tnf}
The set of measures $\mu \in \IRS_{i}(\F_r)$ such that $\F_r \cc (\Sub(\F_r),\mu)$ is totally non-free is residual in $\IRS_{i}(\F_r)$.
\end{cor}


\section{The simplex of invariant measures}
 

\begin{thm}\label{thm:topology}
If $\F_r$ is a nonabelian free group, then $\IRS_{i}(\F_r)$ is a Poulsen simplex. In particular, $\IRS_{ie}(\F_r)$ is a dense $G_\delta$ subset of $\IRS_{i}(\F_r)$ and $\IRS_{ie}(\F_r)$ is homeomorphic to the Hilbert space $l^2$. 
\end{thm}

\begin{proof}
Let $\eta \in  \IRS_i(\F_r)$. It suffices to construct measures $\eta_p \in \IRS_{ie}(\F_r)$ for $p \in (0,1)$ such that $\eta_p$ limits on $\eta$ as $p\searrow 0$. First, we construct a random rooted digraph $\Gamma$ with edge-labels in $S\cup \{*\}$ where $S=\{s_1,\ldots, s_r\}$ is a free generating set of $\F_r$. We will prove that the law of $\Gamma$ is invariant under root-changes and is ergodic with respect to a natural equivalence relation. Then we construct a Schreier coset graph $\phi(\Gamma)$ directly from $\Gamma$. We denote the law of the subgroup corresponding to $\phi(\Gamma)$ by $\eta_p$. From properties of $\Gamma$,  we verify that $\eta_p$ is invariant, ergodic and $\lim_{p \to 0+} \eta_p=\eta$.

\noindent {\bf Construction of $\Gamma$}.

The random rooted graph $\Gamma$ will be constructed as a union $\Gamma= \cup_{n=0}^\infty \Gamma_n$ of increasing subgraphs. The definition of $\Gamma_n$ is recursive. 
Let $K <\F_r$ be a random subgroup with law $\eta$. Let $\Gamma_0=\Sch(K,S)$ denote the Schreier coset graph of $K$. Let $X_0 \subset V(\Gamma_0)$ be a $p$-Bernoulli percolation. Precisely, $X_0$ is a random subset satisfying
\begin{itemize}
\item for each $v \in V(\Gamma_0)$, the probability that $v \in X_0$ equals $p$,
\item the events $\{v \in X_0:~v \in V(\Gamma_0)\}$ are independent.
\end{itemize}
For each $x\in X_0$, let $K_x <\F_r$ be a random subgroup with law $\eta$. We require that the random variables $\{K_x:~x\in X_0\}$ are independent. Let $\Gamma_1$ be the rooted labeled di-graph obtained from the disjoint union of $\Gamma_0$ and $\sqcup_{x\in X_0} \Sch(K_x,S)$ by adding an edge from $x \in X_0$ to $K_x \in V(\Sch(K_x,S))$ for every $x\in X_0$. Each new edge is labeled $*$ and is undirected. The root of $\Gamma_1$ is the same as the root of $\Gamma_0$, namely $K$. See figure \ref{fig:ergodic}.

Assuming $\Gamma_n$ has been constructed (for some $n\ge 1$), let $X_{n}$ be a Bernoulli $p$-percolation on $V(\Gamma_{n})\setminus V(\Gamma_{n-1})$. For each $x \in X_{n}$, let $K_x \in \Sub(\F_r)$ be random with law $\eta$ so that the variables $\{K_x:~x\in X_n\}$ are independent and independent of all other variables of the construction. Construct $\Gamma_{n+1}$ from the disjoint union of $\Gamma_{n}$ and $\sqcup_{x\in X_n} \Sch(K_x,S)$ by adding an undirected $*$-labeled edge from $x \in X_n$ to $K_x \in V(\Sch(K_x,S))$ for each $x\in X_n$. Finally, let $\Gamma  = \cup_{n\ge 0} \Gamma_n$. 

 \begin{figure}[htb]
\begin{center}
\ \psfig{file=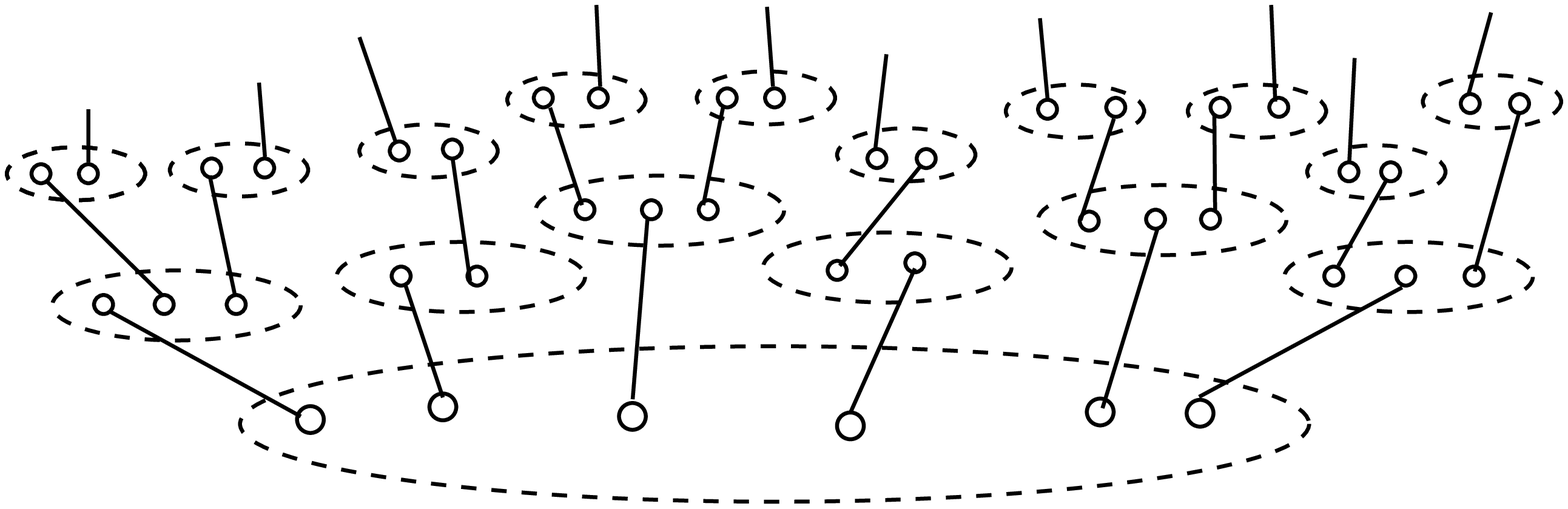,height=1.5in,width=4in}
\caption{The bottom ellipse represents $\Gamma_0$. The circles in the bottom ellipse represent $X_0$. All the edges shown are $*$-labeled edges. The other ellipses represent other $\F_r$-orbits of the vertex set of $\Gamma$ in the sense explained below.}
\label{fig:ergodic}
\end{center}
\end{figure}

\noindent {\bf Construction of $\phi(\Gamma)$}.

Observe that there is a natural right-action of $\F_r$ on the vertex set of $\Gamma$: if $v\in V(\Gamma)$ and $s\in S$ then there is a unique vertex $w \in V(\Gamma)$ such that the edge $(v,w)$ is $s$-labeled. So we define $vs:=w$. This determines a right-action of $\F_r$ on $V(\Gamma)$. 

Let $\phi(\Gamma)$ be the Schreier coset graph obtained from $\Gamma$ as follows. For every $*$-labeled edge $e=\{v,w\}$ of $\Gamma$, we remove $e$, the edge $(v,vs_1)$ and the edge $(w,ws_1)$ and add in the directed $s_1$-labeled edges $(v,ws_1), (w,vs_1)$. See Figure \ref{fig:surgery}. The root of $\phi(\Gamma)$ is the same as the root of $\Gamma$, namely $K$. It is clear that $\Gamma$ is a Schreier coset graph (that is, every vertex has exactly one outgoing $s$-labeled edge and one incoming $s$-labeled edge for every $s\in S$).

 \begin{figure}[htb]
\begin{center}
\ \psfig{file=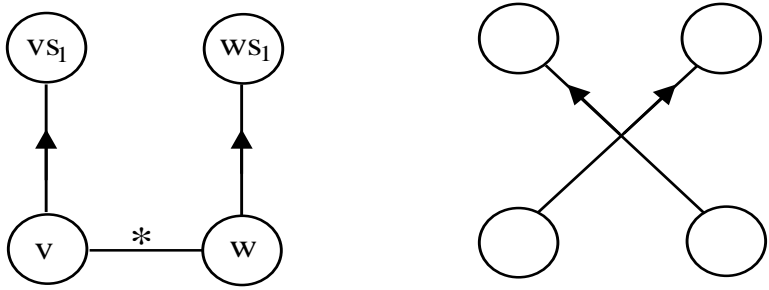,height=1.5in,width=4in}
\caption{On the left is shown a $*$-labeled edge of $\Gamma$ from $v$ to $w$ and $s_1$-labeled edges $(v,vs_1)$ and $(w,ws_1)$. On the right, the corresponding subgraph of $\phi(\Gamma)$ is shown. }
\label{fig:surgery}
\end{center}
\end{figure}

\noindent {\bf Invariance and ergodicity}.

Let $\eta_p$ be the law of the subgroup corresponding to $\phi(\Gamma)$. It now suffices to show that $\eta_p$ is invariant, ergodic and $\lim_{p\to 0+} \eta_p=\eta$. To simplify notation, let $\RG$ denote $\RG(S \cup \{*\}, 2r+1)$ where  $\RG(S\cup \{*\},2r+1)$ is as defined in \S \ref{sec:Schreier}. Observe that $\Gamma$ is a random element of $\RG$. So its law, denoted $\mu$, is a probability measure on $\RG$.




Let $E_\RG \subset \RG \times \RG$ be the isomorphism equivalence relation. Precisely, $(\Lambda,\Lambda') \in E_\RG$ if $\Lambda$ and $\Lambda'$ are isomorphic but not necessarily root-isomorphic.  Also $\mu$ is $E_\RG$-invariant because the construction of $\Gamma$ is independent of the root. The reader might wonder why we use $\cup_n \Gamma_n$ instead of simply $\Gamma_1$. The answer is that the law of $\Gamma_1$ is not $E_\RG$-invariant. Indeed $V(\Gamma_0)$ is the unique $G$-orbit of $V(\Gamma_1)$ that is incident to infinitely many $*$-labeled edges of $\Gamma_1$. So if every vertex of $\Gamma_1$ sends one unit of mass to the closest vertex in $\Gamma_0$ then every vertex of $X_0$ receives infinite mass, in violation of the mass transport principle. 



Next, we show that $\mu$ is $E_\RG$-ergodic. For $\Lambda$ in the support of $\mu$, let $\pi(\Lambda) \in \Sub(\F_r)$ be the subgroup whose Schreier coset graph is the same as the connected component of the root in the graph obtained from $\Lambda$ by removing all of the $*$-labeled edges. Then $\pi_*\mu=\eta$. We claim that if $f$ is any bounded Borel $E_\RG$-invariant function on $\RG$ then there exists a bounded Borel $\F_r$-invariant function $f'$ on $\Sub(\F_r)$ such that $f = f' \circ \pi$ $\mu$-a.e. Indeed, this holds because the construction of $\Gamma$ is Bernoulli conditioned on $K \in \Sub(\F_r)$. That is to say, the random variables $X_0$, $\{K_x:~x\in X_0\}$, $X_1$, etc are all independent. 

Now let $f$ and $f'$ be as above. Because $f$ is $E_\RG$-invariant, if $\Gamma_x$ is isomorphic to $\Gamma$ rooted at $K_x$ (for some $x\in X_0$) then $f(\Gamma_x)=f(\Gamma)$ a.s. Thus $f'(K_x) = f'(K)$ for every $x\in X_0$ a.s. Because $X_0$ is infinite a.s. and $\{K_x:~x\in X_0\}$ are iid with law $\eta$, it follows that $f'$ is constant on $\Sub(\F_r)$ $\eta$-a.e. Therefore $f$ is constant on $\RG$ $\mu$-a.e. Because $f$ is arbitrary, $\mu$ is $E_\RG$-ergodic.

Given $\Lambda$ in the support of $\mu$, let $\phi(\Lambda) \in \SC(\F_r,S)$ be the associated Schreier coset graph as in the definition of $\phi(\Gamma)$. The construction of $\phi(\Gamma)$ from $\Gamma$ shows that for any $g\in \F_r$, there exists $g^\phi \in [E_\RG]$ such that $\phi(g^\phi \Lambda) = g \phi(\Lambda)$ for $\mu$-a.e. $\Lambda$. Because $\mu$ is $E_\RG$-invariant this shows that $\phi_*\mu$ is $\F_r$-invariant. Moreover, if $A \subset \SC(\F_r,S)$ in the image of $\phi$ is any $\F_r$-invariant Borel set then $\phi^{-1}(A)$ is $E_\RG$-invariant. Because $\mu$ is $E_\RG$-ergodic, $\phi_*\mu(A) = \mu(\phi^{-1}(A)) \in \{0,1\}$. Since $A$ is arbitrary, $\phi_*\mu$ is $\F_r$-ergodic.

Because $\eta_p$ is the image of $\phi_*\mu$ under the canonical map from $\SC(\F_r,S)$ to $\Sub(\F_r)$, it follows that $\eta_p$ is also $\F_r$-invariant and ergodic.

\noindent {\bf Continuity}. 

It now suffices to show that $\lim_{p\to 0+} \eta_p=\eta$. Let $F \subset \F_r$ and $r>0$. Let 
$$C(F,r)=\{H \in \Sub(\F_r):~ H \cap B(r) = F\}$$
where $B(r)$ is the ball of radius $r$ centered at the identity in $\F_r$. The span of the characteristic functions of sets of the form $C(F,r)$ is dense in the Banach space of continuous functions on $\Sub(\F_r)$. So it suffices to show that $\lim_{p\to 0+} \eta_p(C(F,r)) = \eta(C(F,r))$ for all $F,r$.

Let $\epsilon>0$. If $p>0$ is sufficiently small then with probability $>1-\epsilon$ the ball of radius $r$ centered at the root of $\Gamma$ has trivial intersection with $X_0$. Conditioned on this event, the ball of radius $r$ centered at the root of $\phi(\Gamma)$ is isomorphic to the ball of radius $r$ centered at the root of $\Sch(K,S)$. Because $K$ is $\eta$-random, it follows that for any $F \subset \F_r$, 
$$\Big|\eta_p(C(F,r)) - \eta(C(F,r)) \Big| \le 2\epsilon.$$
Because $\epsilon,F,r$ are arbitrary, $\lim_{p \to 0+}\eta_p =\eta$.
 
\end{proof}




\section{Generic graphings}

Let $(X,\mu)$ be a standard Borel probability space and $E \subset X \times X$ be a discrete probability-measure-preserving (pmp) Borel equivalence relation. We assume that $E$ is ergodic and aperiodic (which means that a.e. equivalence class is infinite). 

  To set notation, let $\Hom(\F_r,[E])$ denote the set of all homomorphisms from the rank $r$ free group $\F_r=\langle s_1,\ldots, s_r\rangle$ to $[E]$. We identify $\Hom(\F_r,[E])$ with the product space $[E]^r$ in the obvious way and endow it with the product of uniform topologies. Thus it is a Polish space. For latter purposes it will be useful to have the following explicit metric on $\Hom(\F_r,[E])$:
 $${\bf d}_u(\alpha,\beta):= \sum_{i=1}^r d_u(\alpha(s_i),\beta(s_i))$$
 for $\alpha,\beta \in \Hom(\F_r,[E])$. Also let $E_\beta \subset E$ be the subequivalence relation generated by $\beta$. Precisely, $(x,y) \in E_\beta$ if there exists $w\in \F_r$ such that $\beta(w)x=y$.

\begin{defn}\label{defn:graphing}
Let $[[E]]$ be the set of all Borel isomorphisms $\phi:A \to B$ such that $A, B \subset X$ are Borel and the graph of $\phi$ is contained in $E$. A subset $\Gamma \subset [[E]]$ is a {\em graphing} if for a.e. $x\in X$ and every $y$ with $(x,y) \in E$ there is a sequence $\gamma_1,\ldots, \gamma_n \in \Gamma$ and $\epsilon_1, \ldots, \epsilon_n \in \{-1,1\}$ such that $\gamma_1^{\epsilon_1}\cdots \gamma_n^{\epsilon_n} x = y$. The {\em cost} of $\Gamma$ is defined by
$$\textrm{cost}(\Gamma) = \sum_{\gamma \in \Gamma} \mu(\textrm{domain}(\gamma)).$$
The {\em cost} of $E$ is the infimum of cost$(\Gamma)$ over all graphings $\Gamma$ of $E$.
\end{defn}

\begin{lem}\label{lem:graphing0}
Let $\Hom_g(\F_r,[E]) \subset \Hom(\F_r,[E])$ be the set of all homomorphisms $\alpha$ such that $\{\alpha(s_i):~1\le i \le r\}$ is a graphing of $E$ (i.e., $E_\alpha = E$ up to measure zero). If either $r>cost(E)$ or ($r=cost(E)$ and $E$ is treeable) then $\Hom_g(\F_r,[E])$ is a non-empty $G_\delta$-subset of $\Hom(\F_r,[E])$. 
\end{lem}
  
 \begin{proof}
 By \cite[Proposition 1.1]{Hj},  $\Hom_g(\F_r,[E])$ is nonempty. So let $\beta \in \Hom_g(\F_r,[E])$. For $w \in \F_r$ and $\epsilon>0$, let $O(w,\epsilon)$ be the set of all $\gamma \in \Hom(\F_r,[E])$ such that 
 $$\mu(\{x \in X:~ (x,\beta(w)x) \in E_\gamma\}) > \epsilon.$$
 This set is open and
 $$\Hom_g(\F_r,[E]) = \bigcap_{w\in \F_r} \bigcap_{n=1}^\infty O(w,1-1/n).$$
 This shows $\Hom_g(\F_r,[E])$ is a $G_\delta$. 
 \end{proof}

An element $f \in [E]$ is {\em aperiodic} if a.e. $f$-orbit in $X$ is infinite. For somewhat technical reasons (Theorem \ref{thm:isom}), we will be interested in the set $\Hom'_g(\F_r,[E])$ of all homomorphisms $\alpha \in \Hom_g(\F_r,[E])$ such that $\alpha(s_1)$ is ergodic and aperiodic.

\begin{lem}\label{lem:graphing1}
If $E$ is ergodic, aperiodic and $cost(E)<r$ then the set $\Hom'_g(\F_r,[E])$ is a nonempty $G_\delta$ subset of $\Hom(\F_r,[E])$.
\end{lem}

\begin{proof}
The fact that $\Hom'_g(\F_r,[E])$ is nonempty can be seen from the proof of \cite[Lemma 27.7]{KM04} (just choose $\varphi_1$ to be ergodic). That lemma is attributed to Hjorth-Kechris. By Lemma \ref{lem:graphing0}, $\Hom_g(\F_r,[E])$ is a $G_\delta$. Because $E$ is aperiodic, every ergodic element $f\in [E]$ is aperiodic. So it suffices to show that the set of ergodic elements in $[E]$ is a $G_\delta$ subset. This is part of \cite[Theorem 3.6]{K}.
\end{proof} 
  
  

 For $\beta \in \Hom(\F_r,[E])$ let $\Stab^\beta:X \to \Sub(\F_r)$ be the stabilizer map: $\Stab^\beta(x) = \{g \in \F_r:~ \beta(g)x=x\}$. Observe that for any $g\in \F_r$ and $x\in X$, $\Stab^\beta(\beta(g)x) = g \Stab^\beta(x) g^{-1}.$ In other words, $\Stab^\beta$ is $\F_r$-equivariant.  Let $\mu_\beta \in \IRS(\F_r)$ be the pushforward measure $\mu_\beta:=\Stab^\beta_*\mu$.  Therefore, $\Stab^\beta$ defines a factor map from $\F_r \cc^\beta (X,\mu)$ to $\F_r \cc (\Sub(\F_r),\mu_\beta)$.

 \begin{thm}\label{thm:isom}
Let $(X,\mu,E)$ be an ergodic aperiodic discrete pmp equivalence relation. Let $\Hom_{isom}(\F_r,[E])$ be the set of homomorphisms $\beta\in \Hom(\F_r,[E])$ such that $\Stab^\beta$ is an isomorphism from $(X,\mu)$ to $(\Sub(\F_r), \mu_\beta)$. 
 If $r>cost(E)$ then 
$$\Hom_{isom}(\F_r,[E]) \cap  \Hom'_g(\F_r,[E])$$
 is a dense $G_\delta$ subset of $\Hom'_g(\F_r,[E])$. In particular, it is nonempty.
 
  \end{thm}

 \begin{proof}

\noindent {\bf Claim 1}. $\Hom_{isom}(\F_r,[E])$ is a $G_\delta$.

\begin{proof}[Proof of Claim 1] 
Let $\{\xi_i\}_{i=1}^\infty$ be an increasing sequence of finite Borel partitions of $X$ such that $\bigvee_{i=1}^\infty \xi_i$ is the partition into points; i.e. for every $x \ne y \in X$ there exists $i$ such that $x$ and $y$ are in different partition elements of $\xi_i$. 

For $W \subset \F_r$, let $\pi_W: \Sub(\F_r) \to 2^W$ be the projection map $\pi_W(H)=H\cap W$ (where $2^W$ denote the set of all subsets of $W$). For every $n \in \N, \epsilon>0$ and finite  $W\subset \F_r$, let $O(W,n,\epsilon)$ be the set of all $\beta \in \Hom(\F_r,[E])$ such that for every $A,B \in \xi_n$ with $A \ne B$ there exist sets $A' \subset A, B' \subset B$ with 
\begin{itemize}
\item $\mu(A \setminus A') + \mu(B \setminus B') < \epsilon$ and 
\item $\pi_W(\Stab^\beta(A')) \cap \pi_W(\Stab^\beta(B')) = \emptyset$.
\end{itemize}

Each $O(W,n,\epsilon)$ is open in $\Hom(\F_r,[E])$ because for any $\beta \in \Hom(\F_r,[E])$ the set of $\beta' \in \Hom(\F_r,[E])$ which agrees with $\beta$ on $W$ except for a set of measure $<\epsilon$ is open. Also observe that if $W \subset V$ then $O(W,n,\epsilon) \subset O(V,n,\epsilon)$. Thus, $O(n,\epsilon):= \cup_{W \subset \F_r} O(W,n,\epsilon)$ is open. 

We claim that
 $$\Hom_{isom}(\F_r,[E]) = \cap_{n=1}^\infty  \cap_{k=1}^\infty O(n,1/k).$$
The inclusion $\subset$ is obvious. To see the other direction, let $\beta \in \cap_{n=1}^\infty \cap_{k=1}^\infty O(n,1/k)$ and $\mu_\beta = \Stab^\beta_*\mu$. Observe that $\mu_\beta(\Stab^\beta(A)\cap\Stab^\beta(B)) = 0$ for any $A \ne B$ in $\xi_n$ for any $n$. Since 
$$\mu_\beta( \Stab^\beta(A)) = \mu( (\Stab^\beta)^{-1}(\Stab^\beta(A)))$$
this implies $\mu_\beta(\Stab^\beta(A))=\mu(A)$ for every $A \in \xi_n$ for any $n$. Therefore $\beta$ induces a measure-algebra isomorphism from the measure-algebra of $\mu$ to the measure-algebra of $\mu_\beta$. This implies that $\beta$ is a measure-space isomorphism (see e.g. \cite[Corollary B.7]{Zi84}). So $\beta \in \Hom_{isom}(\F_r,[E])$ which proves the equality above. So $\Hom_{isom}(\F_r,[E])$ is a $G_\delta$.
\end{proof}




For $2\le i \le r$, consider the map $\Phi_i:\Hom'_g(\F_r,[E]) \to \Hom'_g(\F_r,[E])$ defined by 
\begin{displaymath}
\Phi_i(\alpha)(s_j)=\left\{ \begin{array}{cc}
\alpha(s_j) & j\ne i\\
\alpha(s_1s_i) & j = i
\end{array}\right.\end{displaymath}
Because $\Phi_i(\alpha)(\F_r)<[E]$ is the same subgroup as $\alpha(\F_r)<[E]$ it follows that $\{\Phi_i(\alpha)(s_j)\}_{j=1}^r$ is a graphing of $E$ if and only if $\{\alpha(s_j)\}_{j=1}^r$ is a graphing. It is easy to check that $\Phi_i$ is a homeomorphism and $\Phi_i$ preserves $\Hom_{isom}(\F_r,[E])$. 

Let $\alpha \in \Hom'_g(\F_r,[E])$. For any measurable subset $Y \subset X$ and $g\in \F_r$, let $\alpha(g)_Y:Y \to Y$ denote the {\em first return time map}. So $\alpha(g)_Y(y) = \alpha(g)^n(y)$ where $n\ge 1$ is the smallest positive integer such that $\alpha(g)^ny \in Y$. By Poincar\'e's Recurrence Theorem, such an $n$ exists for a.e. $y \in Y$.

\noindent {\bf Claim 2}.  After replacing $\alpha$ with $\Phi_i(\alpha)$ (for some $2\le i \le r$) if necessary the following is true. For every $\epsilon>0$ there exists a subset $Y \subset X$ with $1-\epsilon<\mu(Y)<1$ and a generator $s_i$ with $2\le i \le r$ such that $\{\alpha(s_j):~ j \ne i\} \cup \{ \alpha(s_i)_{Y} \}$ is a graphing of $E$.

\begin{proof}[Proof of Claim 2]
For $w\in \F_r$, let $X_w$ be the set of all $x\in X$ such that $\alpha(w)x=x$ and if $w=t_1\cdots t_n$ is the reduced form then there does not exist $2\le j \le n$ such that $\alpha(t_j\cdots t_n)x=x$. 

\noindent {\bf Case 1}. Suppose there exists $w \in \F_r$ such that $\mu(X_w)>0$ and if $w=t_1\cdots t_n$ is in reduced form then $t_n =s_i$ for some $2\le i \le r$ and $t_1 \notin \{t_n,t_n^{-1}\}$. 

We claim that there is a subset $X'_w \subset X_w$ with positive measure such that for every $2 \le j \le n$, 
$$\alpha(t_j\cdots t_n)X'_w \cap X'_w = \emptyset.$$
To prove this, let $F$ be the set of all pairs $(x,y)$ such that $x,y\in X_w$ and $y=\alpha(t_j\cdots t_n)x$ for some $2\le j \le n$. We interpret $F$ as the set of edges of a Borel graph with vertex set $X_w$. This Borel graph has bounded degree (indeed, the degrees are all bounded by $2n$). Also $(x,x) \notin F$ for a.e. $x\in X$ by definition of $X_w$. By \cite[Proposition 4.6]{KST99}, there exists a proper coloring of this graph with only a finite set of colors. This means there is a measurable map $\phi:X_w \to \Omega$ (where $\Omega$ is a finite set) such that for every $(x,y) \in F$, $\phi(x) \ne \phi(y)$. Let $\omega \in \Omega$ be an element in the essential range of $\phi$ and set $X'_w:=\phi^{-1}(\omega)$. Then $X'_w$ satisfies the claim.

Now let $Y \subset X$ be any measurable set with $1-\epsilon<\mu(Y)<1$ such that $Y \cup X'_w = X$. Let $G=\{\alpha(s_j):~ j \ne i\} \cup \{ \alpha(s_i)_{Y} \}$ and $\cR_G \subset X \times X$ be the equivalence relation generated by $G$. Because $\alpha \in \Hom_g(\F_r,[E])$, in order to prove that $\cR_G=E$ (up to a measure zero set) it suffices to show that for any $y \in Y$ and $x \in X \setminus Y$, $(y,\alpha(s_i)y)$ and $(x,\alpha(s_i)x)$ are in $\cR_G$. 

If $\alpha(s_i)y \in Y$ then it is immediate that $(y,\alpha(s_i)y) \in \cR_G$. If not, then let $m\ge 1$ be the smallest number such that $\alpha(s_i)^my \in Y$. Because $\alpha(s_i)y \notin Y$, $\alpha(s_i)y\in X'_w$. So $\alpha(s_i)^2y = \alpha(t_n)\alpha(s_i)y \notin X'_w$. So $\alpha(s_i)^2y \in Y$ and $m=2$. In particular, $(y, \alpha(s_i)^2y) \in \cR_G$ since $\alpha(s_i)^2y = \alpha(s_i)_Yy$. By choice of $X'_w$, 
$$\alpha(t_j\cdots t_{n-1}) \alpha(s_i)^2y = \alpha(t_j\cdots t_n)\alpha(s_i)y \in Y$$
for every $2\le j \le n$. Note
$$(\alpha(t_j\cdots t_n)\alpha(s_i)y, \alpha(t_{j-1}t_j\cdots t_n)\alpha(s_i)y) \in G$$
for all $2\le j \le n$. The case $j=2$ holds because $t_1 \notin \{t_n,t_n^{-1}\}=\{s_i,s_i^{-1}\}$. Because $\alpha(t_1\cdots t_n)\alpha(s_i)y = \alpha(s_i)y$, we have shown that $(y,\alpha(s_i)y)\in \cR_G$ as required.

Now let $x\in X\setminus Y$. Because $x\in X'_w$, $\alpha(t_j\cdots t_n)x \notin X'_w$ for $2\le j\le n$.  So
$$(\alpha(t_j\cdots t_n)x, \alpha(t_{j-1}t_j\cdots t_n)x) \in G$$
for all $2\le j \le n$. The case $j=2$ holds because $t_1 \notin \{t_n,t_n^{-1}\}=\{s_i,s_i^{-1}\}$. This implies $(\alpha(t_1\cdots t_n)x, \alpha(t_n)x) \in \cR_G$. Because $\alpha(t_1\cdots t_n)x=x$ and $\alpha(t_nx)=\alpha(s_i)x$, this finishes case 1. 

\noindent {\bf Case 2}. Suppose there does not exist an element $w\in \F_r$ as in Case 1. For the sake of clarity, let $X^\alpha_w$ denote the set we previously denoted by $X_w$. Because cost($E)<r$, and $\{\alpha(s_i)\}_{i=1}^r$ is a graphing, there exists a non-identity element $w\in \F_r$ such that $\mu(X^\alpha_w)>0$. Observe that if $w'$ is a cyclic conjugate of $w$ then $\mu(X^\alpha_{w'})>0$ since $X^\alpha_{w'}=\alpha(g)X^\alpha_w$ for some $g\in \F_r$.

Observe that $w$ is not a power of $s_1$ because $\alpha(s_1)$ is aperiodic. Suppose $w= t_1\cdots t_n$ is in reduced form. Because no cyclic conjugate of $w$ can be as in Case 1, we must have that $t_j = t_{j+1}$ for all $j$. So $w=s_i^n$ for some $2\le i \le r$ and $n \ne 0$. Without loss of generality we may assume $n$ is the smallest positive integer such that if $w=s_i^n$ then $\mu(X^\alpha_w)>0$. Let $v=(s_1^{-1}s_i)^n$. Observe that $X^{\Phi_i(\alpha)}_v \subset X^\alpha_w$. If these sets are not essentially equal then there is a positive measure set of $x$'s such that $\Phi(\alpha)((s_1^{-1}s_i)^ms_1^{-1})x=x$ for some $0\le m < n$. Suppose this is the case and let $u = (s_1^{-1}s_i)^ms_1^{-1}$. Then $\mu(X^{\Phi_i(\alpha)}_u)>0$. Let $t=s_i^ms_1^{-1}$. Observe that $X^\alpha_t \supset X^{\Phi_i(\alpha)}_u$. So $\mu(X^\alpha_t)>0$ and we are back in Case 1 with $t$ in place of $w$.

On the other hand if $X^{\Phi_i(\alpha)}_v = X^\alpha_w$ (up to a measure zero set) then $\mu(X^{\Phi_i(\alpha)}_v)>0$. By replacing $\alpha$ with $\Phi_i(\alpha)$ and $w$ with $v$ we are back in Case 1.

\end{proof} 
 
Let $\epsilon>0$ and let $Y \subset X$ be as in Claim 2 (after replacing $\alpha$ with $\Phi_i(\alpha)$ if necessary). Without loss of generality, we may assume that the number $i$ from Claim 2 equals $r$. Let $Y_j$ be the set of all $y \in Y$ such that $\alpha(s_r)_Y^j (y) = y$ and $1\le j <\infty$ is the smallest positive number with this property. So $Y=\cup_{j=1}^\infty Y_j$ is a partition of $Y$. Moreover there exists an $N>0$ such that if $Z= \cup_{j \ge N} Y_j$ then $\mu(Z)<\epsilon$. 

By \cite[Lemma 8.5]{K} there exists a hyperfinite aperiodic subequivalence relation $F<E\upharpoonright Z$ such that $\alpha(s_r)_Y \in [F]$. Let $f_\infty \in [F]$ be a generator. This means that $F=\{ (z, f_\infty^nz):~z \in Z, n \in \Z\}$.

Because $\alpha(s_1)$ is ergodic, $\alpha(s_1)_{X\setminus Y}$ is also ergodic. By Rohlin's Theorem, there exists a countable Borel partition $\gamma$ of $X\setminus Y$ such that if $\cB$ is the smallest $\alpha(s_1)_{X\setminus Y}$-invariant sigma-algebra containing $\gamma$, then $\cB$ is the sigma-algebra of all measurable sets  of $X\setminus Y$ (up to sets of measure zero). 
 
Let $\gamma=\{P_i\}_{i=1}^\infty$ and choose elements $f_i \in [E\upharpoonright P_i]$ such that $f_i$ has period $i+N$. To be precise, we require that for a.e. $x \in P_i$, $f_i^{i+N}x=x$ and if $1\le j < i+N$ then $f_i^jx \ne x$. Such an element $f_i$ exists by \cite[Theorem 3.3]{K}.  

Define $\psi \in [E]$ by
\begin{displaymath}
\psi(x) = \left\{ \begin{array}{cc}
\alpha(s_r)_Y(x) & x \in Y \setminus Z \\
f_\infty(x) & x \in Z \\
f_i(x) & x \in P_i 
\end{array}\right.\end{displaymath}
Define a homomorphism $\beta: \F_r \to [E]$ by $\beta(s_i)=\alpha(s_i)$ if $1\le i < r$ and $\beta(s_r)=\psi$.  We claim that
\begin{enumerate}
\item ${\bf d}_u(\alpha,\beta) := \sum_{i=1}^r d_u(\alpha(s_i), \beta(s_i)) \le 4 \epsilon$; 
\item $\beta \in \Hom'_g(\F_r,[E])$,
\item $\beta \in \Hom_{isom}(\F_r,[E])$.
\end{enumerate}
Observe that because $\alpha \in \Hom'_g(\F_r,[E])$ and $\epsilon>0$ are arbitrary (up to replacing $\alpha$ with $\Phi_i(\alpha)$ for some $i$), this claim implies the theorem.

To see the first statement, note that 
$${\bf d}_u(\alpha,\beta)=d_u(\alpha(s_r),\psi)=\mu(\{x\in X:~\alpha(s_r)x \ne \psi(x)\}).$$
Observe that 
$$\{x\in X:~\alpha(s_r)x \ne \psi(x)\} \subset (X \setminus Y) \cup Z \cup \{x \in Y\setminus Z:~ \alpha(s_r)x \notin Y\setminus Z\}.$$
Since $\alpha(s_r)$ is measure-preserving, 
$$\mu\left(\{x \in Y\setminus Z:~ \alpha(s_r)x \notin Y\setminus Z\}\right) \le \mu(X \setminus (Y \setminus Z)) = \mu(X \setminus Y) + \mu(Z) \le 2\epsilon.$$
This implies (1) above.

 
For the second statement, note that the equivalence relation generated by $\{\beta(s_1),\ldots,\beta(s_r)\}$ contains the graph of $\alpha(s_1),\ldots, \alpha(s_{r-1})$ and $\alpha(s_r)_Y$. The latter is true because $\alpha(s_r)_Y \res Y \setminus Z  = \psi \res Y \setminus Z$ and $f_\infty$ is a generator for the equivalence relation generated by $\alpha(s_r)_Y \res Z$. Because $\{ \alpha(s_1),\ldots, \alpha(s_{r-1}), \alpha(s_r)_Y\}$ is a graphing of $E$ it follows that $\beta$ is a graphing of $E$. So $\beta \in \Hom'_g(\F_r,[E])$.

For $k \ge 1$, let $A_k$ be the set of all subgroups $H \in \Sub(\F_r)$ such that $s_r^k \in H$ and if $1\le j < k$ then $s_r^j \notin H$. Observe that $P_i=(\Stab^\beta)^{-1}(A_{i+N})$ for every $i\ge 1$. Because $X\setminus Y  = \cup_{i=1}^\infty P_i$, it follows that $X\setminus Y$ and $\gamma$ are contained in $(\Stab^\beta)^{-1}(\cF)$ where $\cF$ is the Borel sigma-algebra of $\Sub(\F_r)$. Because $\gamma$ is generating for $\alpha(s_1)_{X\setminus Y}$, it follows that $(\Stab^\beta)^{-1}(\cF)$ contains every Borel subset of $X\setminus Y$ (modulo sets of measure zero). Because $\alpha(s_1)$ is ergodic and $\mu(X\setminus Y)>0$, this implies that $(\Stab^\beta)^{-1}(\cF)$ is the sigma-algebra of all measurable sets of $X$, (up to sets of measure zero). Therefore, $\beta \in \Hom_{isom}(\F_r,[E])$ as claimed.
\end{proof}

The following is an immediate consequence.
\begin{cor}\label{cor:universal}
If $G$ is a free group of rank $r$ and $(X,\mu,E)$ is an ergodic aperiodic discrete pmp equivalence relation with cost(E) $<r$ then there exists an invariant measure $\lambda \in \IRS(G)$ such that $(\Sub(G), \lambda, E_G)$ is isomorphic to $(X,\mu,E)$. Moreover, there is an action $G \cc X$ such that $E=\{(x,gx):~x\in X, g\in G\}$ is the orbit-equivalence relation and if $\Stab:X \to \Sub(G)$ is the stabilizer map $\Stab(x)=\{g\in G:~gx=x\}$ then $\Stab$ is a measure-conjugacy from $G\cc (X,\mu)$ to $G\cc (\Sub(G),\lambda)$.
\end{cor}

\section{Encoding via sub-actions}


\begin{thm}\label{thm:encoding}
Let $\F_r=\langle s_1,\ldots, s_r\rangle$ be the free group of rank $r \ge 2$, $\F'_r<\F_r$ be the (infinite-index) subgroup generated by $\langle s_1^2, s_2, \ldots, s_r\rangle$, $K$ be a finite set,  $\F_r \cc K^{\F_r}$ be the usual action $g x(f):=x(g^{-1}f)$. Let $X \subset K^{\F_r}$ be a closed invariant subspace. Then there exist subspaces $Z \subset Y \subset \Sub(\F_r)$ such that
\begin{enumerate}
\item $Z$ is $\F'_r$-invariant, $Y$ is $\F_r$-invariant,
\item there is a finite set $L \subset \F_r$ (depending only on $K$) such that $Y = \bigcup_{f\in L} fZ$,
\item the action $\F_r \cc X$ is topologically conjugate to the conjugation-action $\F'_r \cc Z$. More precisely, if $\phi:\F_r \to \F'_r$ is the isomorphism determined by $\phi(s_1) = s_1^2, \phi(s_i)=s_i$ for $2\le i \le r$ then there is a homeomorphism $\Psi:X \to Z$ such that $\Psi( f x) = \phi(f) \cdot \Psi(x)$ for all $x\in X$ and $f\in \F_r$.
\item If $\eta$ is an $\F_r$-invariant Borel measure on $X$ then there exists an $\F_r$-invariant Borel measure $\lambda$ on $Y$ such that $\Psi_*\eta = \lambda\res Z$ (in particular, $\lambda(Z)>0$). Moreover, if $\eta$ is finite then $\lambda$ is also finite.
\end{enumerate}
\end{thm}

\begin{proof}
Without loss of generality, $K=\{1,\ldots,n\}$ for some integer $n\ge 1$. Given $x \in K^{\F_r}$, we will define a subgroup $\Psi(x) < \F_r$. The easiest way to understand $\Psi(x)$ is through its Schreier coset graph $\Psi(x) \backslash \F_r$ which is constructed as follows. Start with the Cayley graph of $\F_r$. For every element $g \in \F_r$, subdivide the edge $(g,gs_1)$ and attach a cycle of length $x(g)$ to the new vertex. Every edge of this cycle is labeled $s_2$. For every $3\le i \le r$, place a loop labeled $s_i$ at all the new vertices. Also, place a loop labeled $s_1$ at all the new vertices other than the one subdividing the edge $(g,gs_1)$. An example of $x$ and $\Psi(x) \backslash \F_r$ for $r=2$ is shown in figure \ref{fig:sub-action}.

More formally, $\Psi(x)$ is the subgroup generated by all elements of the form
\begin{enumerate}
\item $gs_1 s_2^t s_1^{-1} g^{-1}$ if $g \in \F'_r$ and $x(\phi^{-1}(g)) = 1$ (for any $t \in \Z$);
\item  $g s_1 s_2^t s^{\pm 1}_i s_2^u s_1^{-1} g^{-1}$ if $g \in \F'_r$ and $x( \phi^{-1}(g)) = k>1$ where $i \ne 2$, $t+u = 0 \mod k$, and $u,t \ne 0 \mod k$;
\item $g s_1 s_i s_1^{-1}g^{-1}$ if $g \in \F'_r$ and $3\le i \le r$.
\end{enumerate}

\begin{figure}[htb]
\begin{center}
\ \psfig{file=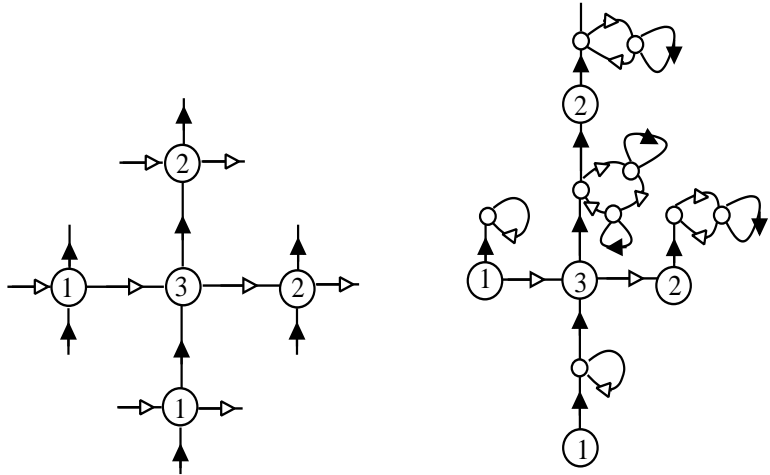,height=2.5in,width=4in}
\caption{The element $x$ is depicted on the left as a function on the vertices of the Cayley graph of $\F_2$. The Schreier coset graph $\Psi(x) \backslash \F_2$ is depicted on the right. Black arrows represent $s_1$ and white arrows represent $s_2$. }
\label{fig:sub-action}
\end{center}
\end{figure}

Clearly, $\Psi( f x) = \phi(f) \cdot \Psi(x)$ for all $x\in X$ and $f\in \F_r$. Note $\Psi:K^{\F_r} \to \Sub(\F_r)$ is a homeomorphism onto its image. Let $Z = \Psi(X)$, so that $\Psi$ gives a homeomorphism from $X$ to $Z$ and let $Y$ be the union of all $\F_r$-orbits in $Z$. From the construction, it is clear that if $L$ is the radius $n$ ball centered at the identity (with respect to the word metric) then $Y = \cup_{f\in L} f \cdot Z$. This, and the fact that $Z$ is closed, implies $Y$ is closed. 

Now suppose $\eta$ is an $\F_r$-invariant Borel measure on $X$. Let $L=\{f_1,\ldots, f_j\}$ be an ordering of $L$ with $f_1=e$. Define $\Upsilon: Y \to Z$ by $\Upsilon( K) =  f_i \cdot K$ where $i$ is the smallest number such that $f_i \cdot K \in Z$. Define a measure $\lambda$ on $Y$ by
$$\lambda(E) = \int |\Upsilon^{-1}(K) \cap E| ~d\Psi_*\eta(K).$$
Clearly, $\lambda\res Z = \Psi_*\eta$ and $\lambda$ is $\F_r$-invariant if and only if $\eta$ is $\F_r$-invariant. Moreover, since $\Upsilon$ is finite-to-1, $\lambda$ is finite if and only if $\eta$ is finite.
\end{proof}

\end{document}